\documentclass[11pt]{amsart}
\usepackage{amsmath}
\usepackage{amssymb}
\usepackage{tabularx}
\usepackage{enumerate}
\usepackage{graphicx}
\usepackage{texdraw}

\usepackage{mathrsfs}

\usepackage{amsfonts,amssymb,amsmath}
\usepackage{epsfig}

\newtheorem{Theorem}{Theorem}[section]
\newtheorem{Lemma}[Theorem]{Lemma}
\newtheorem{Proposition}[Theorem]{Proposition}

\newtheorem{Remark}[Theorem]{Remark}

\newcommand{\R}{\mathbb{R}}

\numberwithin{equation}{section}
\numberwithin{figure}{section}

\begin{document}

\title[Spreading under shifting climate]{A free boundary problem for spreading\\ under shifting climate$^\S$}
\thanks{$^\S$  X. Hao was supported by NSFC (11501318, 11871302), the China Postdoctoral Science Foundation (2017M612230) and the International Cooperation Program of Key Professors by Qufu Normal University,  and Y. Du was supported by the Australian Research Council.
}
\author[Y. Hu, X. Hao, X. Song and Y. Du]{Yuanyang Hu$^1$, Xinan Hao$^2$, Xianfa Song$^3$ and Yihong Du$^4$}
\thanks{$^1$ School of Mathematical Sciences, University of Science and Technology of China, Hefei, Anhui 230026, P.R. China.}
\thanks{$^2$ School of Mathematical Sciences, Qufu Normal University, Qufu 273165, Shandong, P.R. China.}
\thanks{$^3$ Department of Mathematics, School of Mathematics, Tianjin University, Tianjin, 300072, P.R. China.}
\thanks{$^4$ School of Science and Technology, University of New England, Armidale, NSW 2351, Australia.}

\thanks{{\bf Emails:} {\sf yuanyhu@mail.ustc.edu.cn} (Y. Hu),  {\sf haoxinan2004@163.com} (X. Hao), {\sf songxianfa@tju.edu.cn} (X. Song), {\sf ydu@une.edu.au} (Y. Du).}
\date{\today}

\begin{abstract}
In this paper we consider a free boundary problem which models the spreading of an invasive species whose spreading is enhanced by the changing climate. We assume that the climate is shifting with speed $c$ and obtain a complete classification of the long-time dynamical behaviour of the species. The model is similar to that in \cite{DuWeiZhou2018} with a slight refinement in the free boundary condition. While \cite{DuWeiZhou2018}, like many works in the literature, investigates the case that unfavourable environment is shifting into the favourable habitat of the concerned species, here we examine the situation that the unfavourable habitat of an invasive species is replaced by a favourable environment with a shifting speed $c$.
We show that a spreading-vanishing dichotomy holds, and there exists a critical speed $c_0$ such that 
 when spreading happens in the case $c<c_0$, the spreading profile is determined by a semi-wave with forced speed $c$, but when $c\geq c_0$,  the spreading profile is determined by the usual semi-wave with speed $c_0$.
\end{abstract}
\keywords{Diffusive logistic equation, free boundary problem, spreading and vanishing, shifting climate.}

\maketitle

\section{Introduction}

It is widely accepted that climate change has profound impacts on the survival and spreading of  ecological  species.
In recent years, increasing efforts have been devoted to the development and analysis of mathematical models 
 addressing such impacts; see, for example, \cite{BDNZ, BF, DuWeiZhou2018, FLW, HZ, LeiDu2017, LMS, LBBF, LBSF, LWZ, WZ}  and the references therein. 

Most of these models focus on the situation that climate change causes the living environment of the concerned species changing from favourable to unfavourable, and hence endangers the survival of the species. In the case that unfavourable environment shifts into the favourable habitat of the concerned species with speed $c$, a basic feature of the existing models is that there exists a critical speed $c_0$ determined by the favourable environment and the concerned species, such that if $c>c_0$, then the species will vanish eventually, and if $c<c_0$, then the species may survive over a moving band of the environment.

However, some species may benefit from the climate change (\cite{FH}), that is, their living environment is improved by the climate change.
Understanding this kind of effect of climate change is particularly relevant in invasion ecology, as some invasive species may take advantage of the climate change to enhance their invasion. In this paper we consider such a situation based on the model of \cite{DuWeiZhou2018}, which has the following form:
\begin{equation}\left\{\begin{array}{ll}\medskip
\displaystyle u_t=du_{xx}+A(x-ct)u-bu^2,\ \ & t>0,\ \ 0<x<h(t),\\
\medskip\displaystyle u_x(0,t)=u(h(t),t)=0,\ \ & t>0,\\
\medskip\displaystyle h'(t)= -\mu u_x(h(t), t),\ \ & t>0,\\
\displaystyle h(0)=h_0,\; u(x, 0)=u_0(x),\ \ & 0\leq x\leq h_0.
\end{array}\right.\label{MP}
\end{equation}
Here $u(x,t)$ stands for the population density of the concerned species, whose range is the changing interval $[0, h(t)]$;
$h_0,\,\mu, b, c, d$ are positive constants; and
the initial function $u_0(x)$ satisfies
\begin{equation}\label{i}
 u_0\in C^2([0,h_0]),\ u_0'(0)=u_0(h_0)=0, \ u_0'(h_0)<0\ \text{and}\ u_0>0\ \text{in}\ [0,h_0).
\end{equation}

So in this model, the range of the species is the varying interval $[0, h(t)]$, and the species can invade the environment from the right end of the range ($x=h(t)$), with speed proportional to the population gradient $u_x$ there, while at the fixed boundary $x=0$, a no-flux boundary condition is assumed. A deduction of the free boundary condition $h'(t)=-\mu u_x(t, h(t))$ from ecological considerations can be found in \cite{BDu2012}.

The function $A(x-ct)$ represents the assumption that  the environment is changing at a  constant speed $c>0$ in the increasing direction of $x$.
 We assume that $A(\xi)$ is a
Liptschitz continuous function on $\mathbb R^1$ satisfying
 \begin{equation}\label{a}
 A(\xi)=
 \begin{cases}
 a, & \xi\leq 0,\\
 a_0, & \xi\geq l_0,
 \end{cases}
\end{equation}
and $A(\xi)$ is strictly monotone over $[0,l_0]$. Here $l_0$, $a_0$ and $a$ are constants, with $l_0>0$. In \cite{DuWeiZhou2018}, it is assumed (essentially) that $a< 0<a_0$, representing the situation that unfavourable environment is shifting into the favourable habitat of the species with speed $c$. In this paper, we assume instead that
\[
a>0>a_0,
\]
and therefore, the environment is changing from unfavourable to favourable at speed $c$.
Moreover, taking into account that the expanding rate of the population range at its spreading front $x=h(t)$ will most likely vary when the environment changes, we modify \eqref{MP} slightly by changing the free boundary condition $h'(t)=-\mu u_x(h(t),t)$ to
\[
h'(t)=-\mu(A(x-c t))u_x(x, t)|_{x=h(t)},
\]
where the function $\mu(\zeta)\ (\zeta\in [a_0, a])$ is assumed to be continuous and  increasing over $[a_0, a]$,  with
    $$0<\mu(a_0)\leq \mu(a).$$

Thus our revised model in this paper has the following form:		
\begin{equation}\label{free boundary equation}
	\left\{\begin{aligned}
	&u_t=du_{xx}+A(x-ct)u-bu^2,&t&>0,\ 0<x<h(t),\\
	&u_x(0,t)=u(h(t),t)=0,  &t&>0,     \\
	&h'(t)=-\mu(A(h(t)-ct))u_x(h(t),t),&t&>0, \\
	&h(0)=h_0,\ u(x,0)=u_0(x), &0&\leq x \leq h_0.
	\end{aligned}\right.\end{equation}

The case that $A(x-ct)$ and $\mu(A(h(t)-ct))$ in \eqref{free boundary equation} are replaced by positive constants $a$ and $\mu(a)$ respectively, which depicts the spreading of the species in a favourable homogeneous environment, was first studied in \cite{DuLin2010}. When spreading happens, it follows from \cite{BDu2012,DuLin2010,DuMZhou2014} that for $\mu=\mu(a)>0$, there exists a unique $c_0\in (0,2\sqrt{ad})$ such that
$$\lim_{t\to \infty} [h(t)-c_0t]=K$$
for some constant $K$, and
$$\lim_{t\to \infty}\left[\sup_{x\in [0,h(t)]} \left|u(x,t)-q_{c_0}(x-h(t))\right|\right]=0,$$
where $q= q_{c_0}(x)$ is the unique positive solution of
$$\left\{\begin{aligned}
&dq''+c_0q'+aq-bq^2=0,\quad x\in(-\infty,0),\\
& q(0)=0,\ -\mu(a) q'(0)=c_0.
\end{aligned}
\right.$$
We note that $q_{c_0}$ is usually called a semi-wave with speed $c_0$, and it has the following properties:
\[
q_{c_0}(-\infty) =\frac{a}{b},\ \; q'_{c_0}(x)<0 \mbox{ for } x\leq 0.
\]

In order to state our main theorems, apart from the  above defined $c_0$ and $q_{c_0}$, we also need the following result on an auxiliary elliptic problem, which supplies a semi-wave to \eqref{free boundary equation} with speed $c$.
\begin{Proposition}\label{v_L}
	Suppose that $0<c\le c_0$. Then we have the following conclusions:
	
	$(\mathrm{i})$\ For any $L\ge0$, the problem
	\begin{equation}\label{L}
	\left\{
	\begin{array}{l}
-dv''-cv'=A(x)v-b{v}^{2},\quad-\infty<x<L,\\
 v(L)=0
	\end{array}
	\right.	\end{equation}
	has a unique positive solution $v_L(x)$. Moreover,  it  satisfies $v_L(-\infty)=\frac{a}{b}$ and $v_L'(x)<0$ for $x\leq L$.
	
	$(\mathrm{ii})$\ The mapping $L\mapsto v_L'(L)$ is strictly increasing for $L\in [0,+\infty)$.
	
	$(\mathrm{iii})$\ There exists a unique $L_0\geq 0$ such that $-\mu(A(L_0))v_{L_0}'(L_0)=c$. Moreover, $L_0=0$ if and only if $c=c_0$, and in such a case, we have $v_{L_0}=v_0\equiv q_{c_0}$.
\end{Proposition}

We  are now ready to describe the main results of this paper, which completely determine the long-time dynamical behaviour of the unique solution $(u,h)$ of  \eqref{free boundary equation}.

\begin{Theorem}\label{c_0>c}
Assume that $0<c<c_0$. Then one of the following must occur:

${\bf (i)}$ \underline{Vanishing:} $\lim\limits_{t\rightarrow\infty}h(t)=h_\infty<\infty$ and
$$\lim_{t\rightarrow\infty}\|u(\cdot, t)\|_{C([0, h(t)])}=0.$$

${\bf (ii)}$ \underline{Spreading with forced speed:} $\lim\limits_{t\rightarrow\infty}[h(t)-ct]=L_0$, and
$$\lim_{t\to +\infty}\| u(\cdot,t)-v_{L_0}(\cdot+L_0-h(t))\|_{L^{\infty}([0,h(t)])}=0,$$
	where $L_0>0$ and $v_{L_0}(x)$ are given in Proposition 1.1.
	\end{Theorem}

\begin{Theorem}\label{c_0=c}
Suppose that $c=c_0$. Then one of the following must happen:

${\bf (i)}$ \underline{Vanishing:} $\lim\limits_{t\rightarrow\infty}h(t)=h_\infty<\infty$ and
$$\lim_{t\rightarrow\infty}\|u(\cdot, t)\|_{C([0, h(t)])}=0.$$

${\bf (ii)}$ \underline{Spreading:}
There exists a constant $\tilde{C}\le 0$ such that
$$\lim_{t\rightarrow\infty}[h(t)-c_0t]= \tilde{C},$$and
$$\lim_{t\to +\infty}\| u(\cdot,t)-q_{c_0}(\cdot-h(t))\|_{L^{\infty}([0,h(t)])}=0.$$
\end{Theorem}

\begin{Theorem}\label{c_0<c}
Assume that $c>c_0$. Then either

${\bf (i)}$ \underline{Vanishing:} $\lim\limits_{t\rightarrow\infty}h(t)=h_\infty<\infty$ and
$$\lim_{t\rightarrow\infty}\|u(\cdot, t)\|_{C([0, h(t)])}=0,$$or

${\bf (ii)}$ \underline{Spreading:}
There exists $\check{C}\in\R^1$ such that
$$\lim_{t\rightarrow\infty}[h(t)-c_0t]= \check{C},$$and
$$\lim_{t\to +\infty}\| u(\cdot,t)-q_{c_0}(\cdot-h(t))\|_{L^{\infty}([0,h(t)])}=0.$$
\end{Theorem}

\begin{Theorem}\label{u0=sf} Suppose $c>0$. If $h_{0}\ge \frac{\pi}{2}\sqrt{\frac{d}{a}}$, then vanishing cannot happen and hence spreading always happens. If $h_{0}< \frac{\pi}{2}\sqrt{\frac{d}{a}}$ and  $u_0=\sigma\phi$ for some $\phi$ satisfying \eqref{i}, then there exists $\sigma_0=\sigma_0(h_0,\phi,c) \in (0,+\infty]$ such that vanishing happens if and only if $0<\sigma\le\sigma_0$.
	\end{Theorem}
\begin{Remark}
Whether $\sigma_0(h_0,\phi,c) =+\infty$ can actually happen is an open problem.  Some partial answers to this question can be found in \cite{GZ}, where certain sufficient conditions for $\sigma_0(h_0,\phi,c) <+\infty$ are given. For nonlinearities other than the logistic type used in \eqref{free boundary equation}, an example for $\sigma_0(h_0,\phi,c) =+\infty$ can be found in \cite{DuLou}.
\end{Remark}

The rest of the paper is organized as follows. In section 2, we collect some basic results on \eqref{free boundary equation}, which can be proved by similar arguments to those in the existing literature, and we also give the proof of Proposition 1.1, which gives the semi-wave with forced speed $c$, and plays a key role in the long-time behaviour of \eqref{free boundary equation} for the case $c<c_0$.
Sections 3, 4, 5 and 6 are devoted to the proofs of Theorems 1.2, 1.3, 1.4 and 1.5, respectively.
\section{Some basic results}

\subsection{Existence and uniqueness}

The following local existence and uniqueness result can be proved as in \cite{DuLin2010} (see \cite{Wang2019} for some corrections).

\begin{Theorem}\label{local existence}
For any given $u_0$ satisfying \eqref{i} and any $\alpha \in (0,1)$, there is a $T>0$ such that problem \eqref{free boundary equation} admits a unique  solution
$$(u,h)\in C^{1+\alpha,\frac{1+{\alpha }}{2}}(D_T)\times C^{1+\frac{\alpha}{2}}([0,T]);$$
furthermore,
$$\| u \|_{C^{1+\alpha,\frac{1+{\alpha }}{2}}(D_T)}+\|h\|_{C^{1+\frac{\alpha}{2}}([0,T])}\leq C_0,$$
where $D_T=\{(x,t)\in \mathbb{R}^2: x\in [0,h(t)],\ t\in [0,T]\},\ C_0$ and $T$ only depend on $h_0, \alpha$ and $\|u_0\|_{C^2([0,h_0])}$.
\end{Theorem}

To prove that the local solution acquired in Theorem \ref{local existence} can be extended to all $t>0$, we need the following estimates.

\begin{Lemma}\label{estimate}
Let $(u,h)$ be a  solution to problem \eqref{free boundary equation} defined for $t\in (0,T_0)$ for some $T_0\in (0,+\infty]$. Then there exist constants $C_1$ and $C_2$ independent of $T_0$ such that
$$0<u(x,t)\leq C_1,\ 0<h'(t)\leq C_2\ \text{for}\ 0\leq x<h(t)\ \text{and}\ t\in (0,T_0).$$
\end{Lemma}

The proof of Lemma \ref{estimate} is similar to the corresponding result in \cite{DuLin2010}. It follows from $h'(t)>0$ that $h_{\infty}:=\lim\limits_{t\to +\infty}h(t)\in (h_0,+\infty]$ is well defined.

Combining Theorem \ref{local existence} with Lemma \ref{estimate}, as in \cite{DuLin2010}, we obtain the following global existence result.

\begin{Theorem}
The solution of problem \eqref{free boundary equation} is defined for all $t\in (0,\infty)$.
\end{Theorem}

\subsection{Comparison principle}

We give a comparison principle for the free boundary problem, which can be proved similarly as Lemma 3.5 in \cite{DuLin2010}.

\begin{Theorem}\label{comparison principle}
	Suppose that $T\in (0,+\infty),\ \overline{h}\in C^1([0,T]),\ \overline{u}\in C(\overline{D}^{*}_T)\cap C^{2,1}(D^{*}_T)$ with
 $D^{*}_T=\{(x,t)\in \mathbb{R}^2: 0<x<\overline{h}(t),\ 0<t\leq T\}$, and
$$	\left\{
	\begin{aligned}
	&\overline{u}_t\geq d\overline{u}_{xx}+A(x-ct)\overline{u}-b\overline{u}^2,&0&<t\leq T,\ 0<x<\overline{h}(t),\\
	&\overline{u}_x(0,t)\leq 0,\ \overline{u}(\overline{h}(t),t)=0,  &0&<t\leq T,     \\
	&\overline{h}'(t)\geq -\mu(A(\overline{h}(t)-ct))\overline{u}_x(\overline{h}(t),t), &0&<t\leq T.
	\end{aligned}
	\right.	$$
If $$h(0)\leq \overline{h}_0 \quad \mbox{and}\quad u_0(x)\leq \overline{u}(x,0) \ \mbox{in} \  [0,h_0],$$
then the solution $(u,h)$ of the problem \eqref{free boundary equation} satisfies
$$h(t)\leq \overline{h}(t)\ \mbox{in}\ (0,T],\quad u(x,t)\leq \overline{u}(x,t)\ \mbox{for}\ x\in (0,h(t))\ \mbox{and}\ t\in (0,T].$$
\end{Theorem}

\begin{Remark} $(\overline{u},\overline{h})$ is called a supersolution (or an upper solution) to problem \eqref{free boundary equation}. A subsolution (or a lower solution) can be defined analogously by reserving all the inequalities, and a similar comparison principle holds. Theorem \ref{comparison principle} has a few obvious variations with similar proofs, which may also be used in the paper.
\end{Remark}

\subsection{The case of vanishing}

\begin{Theorem}\label{0xiaoyucvanishing}
Suppose that $c>0$. If $h_{\infty}<+\infty$, then
	$$\lim_{t\to +\infty}\|u(\cdot,t)\|_{C([0,h(t)])}=0.$$
\end{Theorem}

\begin{proof}
Due to $h_{\infty}<+\infty$, there exists $\hat{T}>0$ such that $h(t)<ct$ for $t>\hat{T}$. Therefore, for $t>\hat{T}$,
$$A(x-ct)=a\ \mbox{for}\ x\in[0,h(t)],\quad \mbox{and}\quad \mu(A(h(t)-ct))=\mu(a).$$ By the argument in Lemma 3.1 of \cite{DuLin2010}, we deduce that $$\lim_{t\to +\infty}\|u(\cdot,t)\|_{C([0,h(t)])}=0.$$\end{proof}

\subsection{Proof of Proposition 1.1}

	Define
	$$	\underline{v}(x)=\left\{
	\begin{aligned}	
	&q_{c_0}(x),&-&\infty<x<0,\\
	&0,&0&\le x \le L.	
	\end{aligned}\right.$$	
	Due to $0<c\leq c_0$ and $q_{c_0}'(x)< 0$ for $x\leq 0$, it is easily checked that $\underline{v}$ is a lower solution of problem \eqref{L}. Clearly, $\frac{a}{b}$ is an upper solution. By the standard upper and lower solutions argument over an unbounded domain, problem \eqref{L} admits at least one solution $v_L(x)$ satisfying
	$$\underline{v}(x)\le  v_L(x)\le \frac{a}{b}\ \text{in}\ (-\infty,L].$$
For any nontrivial nonnegative solution $V(x)$ of problem \eqref{L}, it follows from the strong maximum principle and Serrin's sweeping argument that $0<V(x)<\frac{a}{b}$ for $x\in (-\infty,L)$.	

We claim that any positive solution $V$ of \eqref{L} satisfies $V(-\infty)=\frac{a}{b}$. Indeed, $$-dV''-cV'=A(x)V-bV^2=aV-bV^2>0\ \text{for}\ x<0,$$
	and hence
	$$(e^{\frac{c}{d}x}V')'= \frac{1}{d}e^{\frac{c}{d}x}(dV''+cV')< 0\ \text{for}\ x<0.$$
	It follows that $e^{\frac{c}{d}x}V'(x)$ is decreasing in $(-\infty,0]$. Since $V$ is bounded, there exists a sequence  $\{x_n\}$ such that $\lim\limits_{n \to +\infty}x_n=-\infty$ and $\lim\limits_{n \to +\infty}V'(x_n)=0$.	Thus,
	$$e^{\frac{c}{d}x}V'(x)< \lim_{n\to +\infty}e^{\frac{c}{d}x_n}V'(x_n)=0\ \text{for}\ x\in (-\infty,0).$$
Hence $V(x)$ is decreasing in $(-\infty,0]$ and $m=\lim\limits_{x \to -\infty}V(x)$ exists. Applying \eqref{L}, we easily obtain $m=\frac{a}{b}$.

	We now prove the uniqueness. Assume that $v_1, v_2$ are two positive solutions of problem \eqref{L}. For any $\epsilon>0$, let $U_i:=(1+\epsilon)v_i,\ i=1,2;$ then it is evident that $$-dU_i''-cU_i'-A(x)U_i+bU_i^2>0\ \text{in}\ (-\infty,L).$$
	Since $\lim\limits_{x\to -\infty}U_i(x)=(1+\epsilon)\frac{a}{b}$ and $\lim\limits_{x\to -\infty}v_i(x)=\frac{a}{b},\ i=1,2$, there exists $l_1>0$ large such that$$U_1(-l)>v_2(-l),\quad U_2(-l)>v_1(-l),\quad  l\ge l_1.$$ It follows from Lemma 2.1 of \cite{DuMa2001} that
	$$(1+\epsilon)v_1(x)\ge v_2(x),\quad (1+\epsilon)v_2(x)\ge v_1(x)\ \mbox{for}\ -l<x\le L,\ l\ge l_1.$$
	Therefore
	$$(1+\epsilon)v_1(x)\ge v_2(x),\quad (1+\epsilon)v_2(x)\ge v_1(x)\ \mbox{for all}\ x\le L.$$
	Letting $\epsilon \to 0$, we deduce that $v_1=v_2$. This implies that $v_L(x)$ is the unique positive solution of problem \eqref{L}.
	
	We next prove conclusion (ii).
	It suffices to prove that $$v_{\bar{L}_1}'(\bar{L}_1)<v_{\bar{L}_2}'(\bar{L}_2)\ \mbox{for}\ 0\le \bar{L}_1<\bar{L}_2.$$	
	Define $v_2(x):=v_{\bar{L}_2}(x+\bar{L}_2-\bar{L}_1)$; then $v_2$ satisfies
	$$\left\{\begin{array}{l}
-dv_2''-cv_2'=A(x+\bar{L}_2-\bar{L}_1)v_2-bv_2^2,\quad -\infty<x<\bar{L}_1,\\
	v_2(-\infty)=\frac{a}{b},\quad v_2(\bar{L}_1)=0.
	\end{array}	\right.	$$
	Due to $A(x+\bar{L}_2-\bar{L}_1)\le A(x)$, by Lemma 2.1 of \cite{DuMa2001}, we obtain
	 $$v_{\bar{L}_1}(x)\ge v_2(x),\quad x\in (-\infty,\bar{L}_1].$$
	Then using strong maximum principle and the Hopf lemma, we conclude that $v_{\bar{L}_1}'(\bar{L}_1)<v_2'(\bar{L}_1)=v_{\bar{L}_2}'(\bar{L}_2)$.
	
	Next, we show that $\lim\limits_{L\to +\infty}v_L'(L)=0$. Define $\phi_L(x):=v_L(x+L)$, then
	\begin{equation*}
	\left\{
	\begin{array}{l}
-d\phi_L''-c\phi_L'=A(x+L)\phi_L-b\phi_L^2,\quad x<0,\\
	\phi_L(-\infty)=\frac{a}{b},\quad \phi_L(0)=0.
	\end{array}
	\right.
	\end{equation*}
	Let $\{L_n\}$ be an increasing sequence that converges to $+\infty$ and $\hat{w}_n(x):=\phi_{L_n}(x)$.
	Applying the comparison principle (see \cite{DuMa2001}), we obtain $$\hat{w}_{n+1}(x)\le\hat{w}_n(x)\ \text{for}\ x\le0.$$
	Owing to $0\le \hat{w}_n(x)\le \frac{a}{b},\ \phi_{\infty}(x):=\lim\limits_{n\to +\infty}\hat{w}_n(x)$ is well-defined on $(-\infty,0]$.
	By $L^p$ theory, Sobolev embeddings, there exists a subsequence of $\{\hat{w}_n\}$ denoted still by $\{\hat{w}_n\}$ such that $\hat{w}_n(x) \to \phi_{\infty}(x)\ \text{in}\ C_{loc}^{1+\alpha}((-\infty,0])$ for some $\alpha\in(0,1)$, and that $\phi_{\infty}(x)$ satisfies
	\begin{equation*}
	\left\{
	\begin{array}{l}
-d\phi_{\infty}''-c\phi_{\infty}'=a_0\phi_{\infty}-b\phi_{\infty}^2,\ \phi_{\infty}\ge 0\ \mbox{for}\ x<0,\\
	\phi_{\infty}(0)=0.
	\end{array}
	\right.
	\end{equation*}
	
	We claim that $\phi_{\infty}(x)\to 0$ as $x\to -\infty$. By  standard regularity consideration, $\phi_{\infty}\in C^{2}((-\infty,0])$ and
	$$d\phi_{\infty}''+c\phi_{\infty}'\ge b\phi^2_{\infty}\ \text{for}\ x<0.$$
	Then
	$$(e^{\frac{c}{d}x}\phi_{\infty}')'\ge \frac{b}{d}e^{\frac{c}{d}x}\phi_{\infty}^2\geq 0\ \text{for}\ x<0.$$
	From this, $e^{\frac{c}{d}x}\phi_{\infty}'(x)$ is nondecreasing in $(-\infty,0]$. As $\phi_{\infty}$ is bounded, we can find a sequence $\{y_n\}$  such that
		$$y_n\to -\infty,\;\; \phi_{\infty}'(y_n)\to 0\ \text{as}\ n\to +\infty.$$
	Therefore
	$$e^{\frac{c}{d}x}\phi_{\infty}'(x)\geq \lim_{n\to +\infty}e^{\frac{c}{d}y_n}\phi'_{\infty}(y_n)=0\ \text{for every}\ x\in (-\infty,0).$$
	Hence, we obtain $\phi_{\infty}'(x)\geq 0\ \text{in}\ (-\infty,0),$ i.e. $\phi_{\infty}(x)$ is a nondecreasing nonnegative function in $(-\infty,0].$
	In view of $\phi_{\infty}(0)=0$, we see that $\phi_{\infty}(x)\equiv 0\ \text{for }\ x\in (-\infty,0]$, and hence
 $$	v_L'(L)\to 0\ \text{as}\ L\to+\infty,$$
\begin{equation}\label{vlt0}-\mu(A(L))v_L'(L)\to 0\ \text{as}\ L\to+\infty.	\end{equation}
Thanks to $0<c\leq c_0$ and $q_{c_0}'(x)<0$ for $x\in(-\infty,0]$, we have	
$$	-dq_{c_0}''-cq_{c_0}' \le-dq_{c_0}''-c_0q_{c_0}' =aq_{c_0}-bq_{c_0}^2.$$
Since $$-dv_0''-cv_0' = av_0-bv_0^2\ \text{for}\  x\in(-\infty,0),$$
$v_0(-\infty)=q_{c_0}(-\infty)=\frac{a}{b}$ and $v_0(0)=q_{c_0}(0)=0$, it follows from the comparison principle and the Hopf boundary lemma that either $v_0\equiv q_{c_0}$, which is possible only if $c=c_0$, or $c<c_0$ and
 $$	v_0'(0)<q_{c_0}'(0)=-\frac{c_0}{\mu(a)}=-\frac{c_0}{\mu(A(0))},$$
  thereby
 \begin{equation}\label{vld-} -\mu(A(0))v_0'(0)>c_0. \end{equation}
 By  \eqref{vlt0} and \eqref{vld-},
  the continuous dependence of $\mu(A(L))v_L'(L)$ on $L$ and the monotonicity of $\mu(A(L))v_L'(L)$ in $L$, there exists a unique $L_0\in (0,+\infty)$ such that $$-\mu(A(L_0))v_{L_0}'(L_0)=c.$$
	This completes the proof.
\begin{Remark}\label{rm2.6}
If $L<0$, then due to $A(x)\equiv a$ for $x\leq 0$, it is easily seen that in this case the unique positive solution of  \eqref{L} is given by $v_L(x)\equiv v_0(x-L)$. In particular, when $c=c_0$, then $v_L(x)\equiv q_{c_0}(x-L)$ for $L<0$.
\end{Remark}

\section{Proof of Theorem \ref{c_0>c}}

This section is devoted to the proof of  Theorem \ref{c_0>c}. 
Clearly the vanishing case (i) follows directly from Theorem \ref{0xiaoyucvanishing}. So we only need to consider the spreading case (ii). Accordingly, unless otherwise specified, we always assume $0<c<c_0$ and $h_{\infty}=+\infty$ in the rest of this section. 
Since the proof is quite long, for clarity, we carry it out in two subsections.
\subsection{Behaviour of $h(t)$}

In this subsection, we completely determine the long-time behaviour  of $h(t)$. 

\begin{Lemma}\label{psM}
For any given $L_1>L_0$ and $l>0$, the problem
	\begin{equation}\label{psi}
		\left\{\begin{array}{l}
-d\psi''-c\psi'=A(x)\psi-b\psi^2,\quad -l<x<L_1,\\
    \psi(-l)=M,\quad \psi(L_1)=0
		\end{array}\right.
	\end{equation}
	has a unique positive solution $\psi_{-l,L_1}(x)$, where $M= \max\{\frac{a}{b},\|u_0\|_{\infty}\}$. Moreover, when $l$ is large enough,
	$$ -\mu(A(L_1))\psi_{-l,L_1}'(L_1)<c.$$
\end{Lemma}

\begin{proof}
It is easy to check that $M$ is a supersolution of problem \eqref{psi}, and $v_{L_1}$ is a subsolution of problem \eqref{psi}. It now follows from the supersolution and subsolution argument, and the comparison principle for logistic equations (see \cite{DuMa2001}) that problem \eqref{psi} admits a unique positive solution $\psi_{-l,L_1}(x)$ satisfying $v_{L_1}(x)\le\psi_{-l,L_1}(x)\le M$ for $x\in [-l,L_1]$.
	
	Let $\{l_n\}$ be an arbitrary increasing sequence satisfying $l_n\to +\infty$ and $l_1>0$. By standard $L^p$ estimates, the Sobolev embedding theorem, and a diagonal process, we can find a subsequence of $\{\psi_{-l_n,L_1}\}$, for simplicity, still represented by itself, such that
	\begin{equation*}\psi_{-l_n,L_1}(x)\to \tilde{v}_{L_1}(x)\ \mbox{in}\ C^{1+\alpha}_{loc}((-\infty,L])\ \mbox{as}\ n\to +\infty,	\end{equation*}
	where $\alpha\in(0,1)$. It is clear that $\tilde{v}_{L_1}(x)$ satisfies
     \begin{equation*}
	  \left\{\begin{array}{l}
-d\tilde{v}_{L_1}''-c\tilde{v}_{L_1}'=A(x)\tilde{v}_{L_1}-b\tilde{v}_{L_1}^2,\quad -\infty<x<L_1,\\
      \tilde{v}_{L_1}(L_1)=0.
	  \end{array} \right.
     \end{equation*}	

By Proposition \ref{v_L}, we see that
$$ \tilde{v}_{L_1}(x)\equiv v_{L_1}(x)\ \mbox{for}\ -\infty<x<L_1.$$
Therefore,
\begin{equation}\label{ps}
\psi_{-l,L_1}(x)\to v_{L_1}(x)\ \mbox{in}\ C^{1+\alpha}_{loc}((-\infty,L_1])\ \mbox{as}\ l\to +\infty.
\end{equation}
Since $L_1>L_0$, we have
$$-\mu(A(L_1))v_{L_1}'(L_1)<-\mu(A(L_0))v_{L_0}'(L_0)=c.$$
Combining this with \eqref{ps}, we conclude that for all sufficiently large $l$,
$$ -\mu(A(L_1)){\psi}_{-l,L_1}'(L_1)<c.$$	
The proof of this lemma is complete.
\end{proof}

\begin{Lemma}\label{-H}$\limsup\limits_{t\to +\infty}[h(t)-ct]<+\infty$.\end{Lemma}

\begin{proof}
According to Lemma \ref{psM}, for $L_1>L_0$, we can choose a large constant $l^0>0$ such that
	$$-\mu(A(L_1))\psi_{-l^0,L_1}'(L_1)<c.$$ Since $\psi_{-l^0,L_1}(x)$ achieves its maximum $M$ at $x=-l^0$, we have $\psi_{-l^0,L_1}'(-l^0)\leq 0$.
	 Choose $r>h_0+l^0$, and define
	\begin{equation*}
	\eta(t):=ct+L_1+r,\;\;	\bar{v}(x):=
	\begin{cases}
	M,&x\in(-\infty,-l^0),\\
	\psi_{-l^0,L_1}(x),&x\in[-l^0,L_1], 	
	\end{cases}
	\end{equation*}
and	 $$\bar{u}(x,t):=\bar{v}(x-ct-r).$$ 	
Then it is easy to check that  for $t>0$ and $x\in[0,\eta(t)]$, in the weak sense,
  	\begin{equation*}
   	\begin{split}
   	\bar{u}_t-d\bar{u}_{xx}&\geq A(x-ct-r)\bar{u}-b\bar{u}^2\\
   	           &\ge A(x-ct)\bar{u}-b\bar{u}^2,\\   	
   	\end{split}    	
   	\end{equation*}
   $$\bar{u}(\eta(t),t)=\bar{v}(L_1)=0,$$
 	\begin{equation*}
   	\begin{split}-\mu(A(\eta(t)-ct))\bar{u}_x(\eta(t),t)&=-\mu(A(L_1+r))\psi_{-l^0,L_1}'(L_1)\\
   &\leq -\mu(A(L_1))\psi_{-l^0,L_1}'(L_1)<c=\eta'(t),\end{split}    	
   	\end{equation*}
  $$\bar{u}_x(0,t)=\bar{v}'(-ct-r)=0,$$
and
$$u_0(x)\le M=\bar{u}(x,0)\ \text{for}\ x\in[0,h_0].$$
On account of the comparison principle, we have
\begin{equation}\label{ht}
	h(t)\le\eta(t)=ct+L_1+r\ \text{for} \ t>0,
\end{equation}
$$u(x,t)\le\bar{u}(x,t)\ \text{for}\ t>0\ \mbox{and}\ 0<x<h(t).$$
From \eqref{ht}, we obtain $\limsup\limits_{t\to +\infty}[h(t)-ct]<+\infty$.
\end{proof}

\begin{Lemma}\label{H-}$\liminf\limits_{t\to +\infty}[h(t)-ct]>-\infty$.\end{Lemma}

\begin{proof} Since $0<c<c_0<2\sqrt{ad},$ for large $l>0$, the problem
	\begin{equation*}
	 \left\{\begin{array}{l}
-dU''-cU'=aU-bU^2,\quad -l<x<0,\\
		U(-l)=U(0)=0
		\end{array} \right.	\end{equation*}
admits a unique positive solution $U_l$ (see the proof of Proposition 2.1 in \cite{BDu2012}).	
For any given $L>0$, let
\begin{equation*}
	\underline{v}(x)=
	 \begin{cases}
	U_l(x),&x\in [-l,0],\\
	0,&x\in (0,L].	
	\end{cases}
\end{equation*}
Then it is easily seen that $\underline{v}(x)$ is a subsolution and $M$ is a supersolution of the problem
\begin{equation}\label{w-lL}
	\left\{\begin{array}{l}
-dw''-cw'=A(x)w-bw^2,\quad -l<x<L,\\
 w(-l)=w(L)=0.
	 \end{array}\right.\end{equation}
Thus, on account of supersolution and subsolution argument, and the comparison principle for logistic equations (see \cite{DuMa2001}), problem \eqref{w-lL} has a unique positive solution $w_{-l,L}(x)$.

Now, we choose $L_2$ satisfying $0<L_2<L_0$. Let $\{l_n\}$ be an arbitrary sequence satisfying $l_n\to +\infty$. Similar to the discussion in the proof of Lemma \ref{psM}, we can find a subsequence of $\{w_{-l_n,L_2}\}$, represented still by
$\{w_{-l_n,L_2}\}$ for the sake of convenience, such that
$$w_{-l_n,L_2}(x)\to v_{L_2}(x)\ \mbox{in}\ C^{1+\alpha}_{loc}((-\infty,L_2])\ \mbox{as}\ n\to +\infty,\ \mbox{where}\ \alpha\in (0,1).$$
It follows that $$w_{-l,L_2}(x)\to v_{L_2}(x)\ \mbox{in}\ C^{1+\alpha}_{loc}((-\infty,L_2])\ \mbox{as}\ l\to +\infty.$$ 	
Due to $-\mu(A(L_2))v_{L_2}'(L_2)>-\mu(A(L_0))v_{L_0}'(L_0)=c$, we have 	
	$$-\mu(A(L_2))w_{-l',L_2}'(L_2)>c$$
for all sufficiently large $l'$.
By the strong maximum principle, there exists $\epsilon>0$ such that $$w_{-l',L_2}(x)\le \frac{a}{b}-\epsilon,\quad x\in [-l',L_2].$$
	
Set $$B(t):=\min_{t\in [0,+\infty)}\{h(t),ct\};$$ due to $h_{\infty}=+\infty,\ \lim\limits_{t\to +\infty}B(t)=+\infty$.
Since $x-ct\leq B(t)-ct\le0$ for $x\in[0,B(t)]$ and
 $t>0,\ A(x-ct)=a$ for such $x$ and $t$.
A slight change (substituting $B(t)$ for $h(t)$) in the proof of Lemma 3.2 in \cite{DuLin2010}  shows that
\begin{equation}\label{cuab}
	\displaystyle\lim_{t\to +\infty}u(x,t)=\frac{a}{b}\ \text{uniformly for}\ x\ \text{in any compact subset of}\ [0,\infty).
\end{equation}

We can select $T_1>0$ such that $ct>l'$ for $t\ge T_1$.
Thanks to $h_{\infty}=+\infty,$ there exists
$T_2>T_1$ such that $h(t)>L_2+cT_1$ for all $t\ge T_2$. We now define $$\eta_1(t):=-l'+ct,\ \eta_2(t):=L_2+ct\ \text{for}\ t>T_1,$$
$$\underline{u}(x,t):=w_{-l',L_2}(x-ct)\quad \text{for}\quad t>T_1,\ x\in[\eta_1(t),\eta_2(t)].$$
From \eqref{cuab}, there exists $T_3>T_2$ such that
$$u(x,T_3)>\frac{a}{b}-\epsilon\quad \text{for}\quad x\in[\eta_1(T_1),\eta_2(T_1)].$$
Explicit computations show that
$$\underline{u}_t-d\underline{u}_{xx}=A(x-ct)\underline{u}-b\underline{u}^2\ \text{for}\ t>T_1,\ x\in[\eta_1(t),\eta_2(t)],$$
$$\eta_2(T_1)=L_2+cT_1<h(T_3),$$
$$\underline{u}(\eta_1(t)),t)=\underline{u}(\eta_{2}(t),t)=0\ \text{for}\ t>T_1,$$
\begin{equation*}\begin{split}-\mu(A(\eta_2(t)-ct))\underline{u}_x(\eta_2(t),t)&=-\mu(A(L_2))w_{-l',L_2}'(L_2)>c=\eta_2'(t)\ \text{for}\ t>T_1,\end{split}\end{equation*}
and$$\underline{u}(x,T_1)<u(x,T_3)\ \text{for}\ x\in[\eta_1(T_1),\eta_2(T_1)].$$
It follows from the comparison principle that
   $$u(x,t+T_3)\geq \underline{u}(x,t+T_1)\ \text{for}\ t>0\ \mbox{and}\ x\in[\eta_1(t+T_1),\eta_2(t+T_1)],$$
   $$h(t+T_3)\ge\eta_2(t+T_1)=L_2+c(t+T_1)\ \text{for all}\ t>0.$$
   This proves the lemma.\end{proof}

Combining Lemmas \ref{-H} and \ref{H-}, we obtain the following result.

\begin{Lemma}\label{bounded}There exists constant $C_3>0$ such that
	$$| h(t)-ct|< C_3\ \text{for all}\ t>0.$$
\end{Lemma}

\begin{Lemma}\label{barH}
	 $\bar{H}:=\limsup\limits_{t\to +\infty}[h(t)-ct]\le L_0$.
\end{Lemma}

\begin{proof}
Suppose the lemma is false; then $\bar{H}>L_0$. By Lemma \ref{bounded}, 
	\begin{equation}\label{C3}
		-C_3\le h(t)-ct \le C_3\ \text{for}\ t>0.
	\end{equation}
Let $t_n\to \infty$ be a positive sequence satisfying $\lim\limits_{n\to +\infty}[h(t_n)-ct_n]=\bar{H}$.
We define $$\hat{g}(t):=h(t)-ct+2C_3,\quad v(x,t):=u(x+ct-2C_3,t)\ \text{for}\ t>0,$$	$$v_n(x,t)=v(x,t+t_n),\quad \hat{g}_n(t)=\hat{g}(t+t_n),$$	
$$y=\frac{x}{\hat{g}_n(t)},\quad w_n(y,t)=v_n(x,t).$$	
Then
\begin{equation}\label{wn}	(w_n)_t-\frac{c+\hat{g}'_n(t)y}{\hat{g}_n(t)}(w_n)_y-\frac{d}{{\hat{g}_n}^2(t)}(w_n)_{yy}=A(y\hat{g}_n(t)-2C_3)w_n-bw_n^2
\end{equation}	
	for $\frac{-c(t+t_n)+2C_3}{\hat{g}_n}\le y < 1,\ t>-t_n$,
	 \begin{equation}\label{1t}
		w_n(1,t)=0\ \text{for}\ t>-t_n,
	 \end{equation}
and	 \begin{equation}\label{wngn}
	 	-\mu(A(\hat{g}_n(t)-2C_3))\frac{(w_n)_y(1,t)}{\hat{g}_n(t)}=\hat{g}_n'(t)+c\ \text{for}\ t>-t_n.
	 \end{equation}
	 Let us observe that, for $t>-t_n$ and $y\in [-X, 0]$ with any fixed $X>0$,
	 \[
	 \frac{d}{{\hat{g}_n}^2(t)} \mbox{ is uniformly continuous, and }  \frac{c+\hat{g}'_n(t)y}{\hat{g}_n(t)} \mbox{ is uniformly bounded}.
	 	 	 \]	 
Therefore, for any given $X>0,\ T_1\in \mathbb{R}^1$, and $p>1$, we can apply the parabolic $L^p$ estimate to \eqref{wn} and \eqref{1t} over $[-X-\frac{1}{2},1]\times[T_1-\frac{1}{2}, T_1+\frac{1}{2}]$, to conclude that there exists a positive integer $N_1$ such that
$$\| w_n\|_{W_p^{2,1}([-X,1]\times[T_1,T_1+\frac{1}{2}])}\le C_4\ \mbox{for all}\ n\ge N_1,$$
for some constant $C_4>0$ which depends on $X$ and $p$ but does not depend on $N_1$ and $T_1$.
Hence, by Sobolev imbeddings, we conclude that there exists $\beta\in (0,1)$ such that
\begin{equation}\label{C5}
	\| w_n\|_{C^{1+\beta,\frac{1+\beta}{2}}([-X,1]\times[T_1,+\infty))}\le C_5\ \text{for every}\ n\ge N_1,
\end{equation}
for some constant $C_5>0$ which depends on $\beta$ and $X$ but does not depend on $N_1$ and $T_1$.
Combining \eqref{wngn} with \eqref{C5}, we obtain
$$	\|\hat{g}_n\|_{C^{1+\frac{\beta}{2}}([T_1+\infty))}\le C_6\ \text{for all sufficiently large}\ n,$$
where $C_6$ is a constant independent of $n$ and $T_1$.
 Thus, there exists a subsequence of $w_n$, denoted still by $w_n$ for convenience, such that
 \begin{equation*}\label{wntoW}
 	w_n\to W\ \text{in}\ C_{loc}^{1+\beta^{\prime},\frac{1+\beta^{\prime}}{2}}((-\infty,1]\times \mathbb{R}^1)\ \mbox{as}\ n\to +\infty
 \end{equation*}
 for some $\beta^{\prime}\in (0,\beta)$. It follows that, subject to passing to a further subsequence,
\begin{equation}\label{gnG}
	\hat{g}_n\to G\ \text{in}\ C_{loc}^{1+\frac{\beta^{\prime}}{2}}(\mathbb{R}^1).
\end{equation}
 By the parabolic Schauder estimate, we know $(W,G)$ satisfies the following equations in the classical sense:
\begin{equation*}\label{re}
\left\{
\begin{aligned}
&W_t-\frac{d}{G^2}W_{yy}-\frac{c+G'(t)y}{G}W_y=A(yG(t)-2C_3)W-bW^2,\ &t&\in \mathbb{R}^1,\ y\in (-\infty,1),\\
&W\left(1,t\right)=0,\ -\mu(A(G(t)-2C_3))\frac{W_y(1,t)}{G}=G'(t)+c,\ &t&\in \mathbb{R}^1. \end{aligned}
\right.
\end{equation*}
Let $x=G(t)y$ and $V(x,t):=W(y,t)$. Then it is easy to check that $(V,G)$ satisfies
\begin{equation*}
\left\{
\begin{aligned}
&V_t-dV_{xx}-cV_x=A(x-2C_3)V-bV^2,\ &t&\in \mathbb{R}^1,\ x\in (-\infty,G(t)),\\
&V\left(G(t),t\right)=0,\ -\mu(A(G(t)-2C_3)) V_x(G(t),t)=G'(t)+c,\ &t&\in \mathbb{R}^1, \end{aligned}
\right.
\end{equation*}
and
$$	v_n\to V\ \text{in}\ C_{loc}^{1+\beta^{\prime},\frac{1+\beta^{\prime}}{2}}((-\infty,G(t)]\times \mathbb{R}^1).$$
Set $\hat{V}(x,t):=V(x+2C_3,t)$ and $F(t):=G(t)-2C_3$, then $(\hat{V},F)$ satisfies
\begin{equation*}\label{jianV}
\left\{
\begin{aligned}
&V_t-dV_{xx}-cV_x=A(x)V-bV^2,\ &t&\in \mathbb{R}^1,\ x\in (-\infty,F(t)),\\
&V\left(F(t),t\right)=0,\ -\mu(A(F(t))) V_x(F(t),t)=F'(t)+c,\ &t&\in \mathbb{R}^1. \end{aligned}
\right.
\end{equation*}
From \eqref{C3} and \eqref{gnG}, we obtain
$$ 
\underline{H}+2C_3\le G(t) \le\bar{H}+2C_3\ \text{for}\ t\in
\mathbb{R}^1,
$$
and hence
$$ 
\underline{H}\le F(t)\le \bar{H}\ \text{for}\ t\in \mathbb{R}^1.
$$
In view of $F(0)=\displaystyle\lim_{n\to +\infty}[h(t_n)-ct_n]=\bar{H}$, we have 
$$
\sup_{t\in
	\mathbb{R}^1}F(t)=F(0)=\bar{H},
	$$
which implies that  $F^{\prime}(0)=0$, and hence 
\[
-\mu(A(\bar{H})) \hat{V}_x(\bar{H},0)=c.
\]

Due to $\bar{H}>L_0$, by  Lemma \ref{psM}, we can find a large constant $\bar{l}>0$ such that
\begin{equation}\label{psibarl}
	-\mu(A(\bar{H}))\psi'_{-\bar{l},\bar{H}}(\bar{H})<c,
\end{equation}
where $\psi_{-\bar{l},\bar{H}}$
is the unique positive solution of
\begin{equation*}
\left\{
\begin{array}{l}
-d\psi''-c\psi'=A(x)\psi-b\psi^2,\quad -\bar{l}<x<\bar{H},\\
\psi(-\bar{l})=M,\quad \psi(\bar{H})=0.
\end{array}
\right.
\end{equation*}

   Consider the solution $\bar{v}_{\bar{H}}$ of
 \begin{equation*} \begin{cases}
 V_t-dV_{xx}-cV_x=A\left(x\right)V-bV^2,\ &t>0,\ -\bar{l}<x<\bar{H},\\
 V(-\bar{l},t)=M,\ V\left(\bar{H},t\right)=0 &t>0,\\
 V(x,0)=M,\ &x\in [-\bar{l},\bar{H}].    	
 \end{cases}
 \end{equation*}
 Since $M$ is a super solution of the corresponding elliptic problem, it follows from a well-known result on parabolic equations that $\bar{v}_{\bar{H}}(x,t)$ is decreasing in $t$ and
\begin{equation}\label{barvtopsi}
\lim_{t\to +\infty}\bar{v}_{\bar{H}}(x,t)=\psi_{-\bar{l},\bar{H}}(x)\quad \text{uniformly for}\ x\in [-\bar{l},\bar{H}].
\end{equation}
Since $0\leq u\leq M$, we have $0\leq \hat V\leq M$, and so we can use
 the comparison principle to conclude that
  $$\bar{v}_{\bar{H}}(x,t)\geq \hat{V}(x,t-s),\quad \forall s>0,\ t>0,\ -\bar{l}<x<F(t).
  $$
 Choosing $s_n\to \infty$, we obtain $\bar{v}_{\bar{H}}(x,s_n)\geq \hat{V}(x,0)$ for $-\bar{l}<x<\bar{H}$.
 Combining this with \eqref{barvtopsi}, we have $$\psi_{-\bar{l},\bar{H}}(x)\geq \hat{V}(x,0)\quad \mbox{for}\ -\bar{l}<x<\bar{H}.$$
It follows from the strong maximum principle and the Hopf boundary lemma that $$\psi'_{-\bar{l},\bar{H}}(\bar{H})<\hat{V}_x(\bar{H},0)=-\frac{c}{\mu(A(\bar{H}))},$$
which contradicts \eqref{psibarl}.
The proof is now complete.\end{proof}

\begin{Lemma}\label{underlineH}
	$\underline{H}:=\liminf\limits_{t\to +\infty}[h(t)-ct]\ge L_0$.
\end{Lemma}

\begin{proof}
Assume by contradiction that $\underline{H}<L_0$. By Lemma \ref{bounded}, 
$-C_3\le h(t)-ct \le C_3$ for $t>0$. Set$$\hat{g}(t):=h(t)-ct+2C_3,\ v(x,t):=u(x+ct-2C_3,t)\ \text{for}\ t>0.$$	
Let $t_n\to \infty$ be a positive sequence satisfying $\lim\limits_{n\to +\infty}[h(t_n)-ct_n]=\underline{H}$. We define $$v_n(x,t)=v(x,t+t_n),\quad \hat{g}_n(t)=\hat{g}(t+t_n).$$	By the same argument as in the proof of Lemma \ref{barH}, we obtain, by passing to a subsequence,
$$\hat{g}_n\to G\ \text{in}\ C_{loc}^{1+\frac{\beta}{2}}(\mathbb{R}^1)\quad \text{and}\quad v_n\to V\ \text{in}\ C_{loc}^{1+\beta,\frac{1+\beta}{2}}((-\infty,G(t)]\times \mathbb{R}^1)$$for some $\beta\in (0,1)$; moreover, 
 $\hat{V}(x,t):=V(x+2C_3,t)$ and $F(t):=G(t)-2C_3$ satisfy
\begin{equation*}\label{jianV}
\left\{
\begin{aligned}
&V_t-dV_{xx}-cV_x=A(x)V-bV^2,\ &t&\in \mathbb{R}^1,\ x\in (-\infty,F(t)),\\
&V\left(F(t),t\right)=0,\ -\mu(A(F(t))) V_x(F(t),t)=F'(t)+c,\ &t&\in \mathbb{R}^1. \end{aligned}
\right.
\end{equation*}
Due to $\underline{H}\le F(t)\le \bar{H}\ \text{for}\ t\in \mathbb{R}^1$ and $F(0)=\lim\limits_{n\to +\infty}[h(t_n)-ct_n]=\underline{H}$, we have  $F^{\prime}(0)=0$, and hence 
\[
-\mu(A(\underline{H})) \hat{V}_x(\underline{H},0)=c.
\]

 Due to $\underline{H}<L_0$, by Proposition 1.1 and Remark \ref{rm2.6} we have $V_{\underline H}(x)=V_{H_*}(x-\underline H+H_*)$, where $H_*=\max\{\underline H, 0\}<L_0$. It follows that
 \[
 -\mu(A(\underline H))V'_{\underline H}(\underline H)=-\mu(A(H_*))V'_{H_*}(H_*)>\mu(A(L_0))V_{L_0}'(L_0)=c.
 \]
 Similar to the proof of Lemma \ref{psM}, we find that for all  large constant $\underline{l}>0$,
\begin{equation}\label{psiunderlinel}
	-\mu(A(\underline{H}))\psi'_{-\underline{l},\underline{H}}(\underline{H})>c,
\end{equation}
where $\psi_{-\underline{l},\underline{H}}$
is the unique positive solution of
\begin{equation*}
\left\{
\begin{array}{l}
-d\psi''-c\psi'=A(x)\psi-b\psi^2,\quad -\underline{l}<x<\underline{H},\\
\psi(-\underline{l})=0,\quad \psi(\underline{H})=0.
\end{array}
\right.
\end{equation*}

On the other hand, from the proof of Lemma \ref{H-}, we have
\[
u(x,t+T_3)\geq \underline u(x, t+T_1)=w_{-l',L_2}(x-c(t+T_1))
\]
for 
\[\mbox{$x\in [\eta_1(t+T_1), \eta_2(t+T_1)]=[-l'+c(t+T_1), L_2+c(t+T_1)]$,\ $t\geq 0$.}
\]
 It follows that
\[
v(x,t)=u(x+ct-2C_3,t)\geq w_{-l', L_2}(x-2C_3+c(T_3-T_1))
\]
for 
\[
x\in [-l'-c(T_3-T_1)+2C_3, L_2-c(T_3-T_1)+2C_3],\; t\geq T_3.
\]
 This implies that
\begin{equation}\label{new}
\hat V(x,t)\geq w_{-l', L_2}(x+c(T_3-T_1)) \mbox{ for } x\in [-l'-c(T_3-T_1), L_2-c(T_3-T_1)],\; t\in\mathbb R^1.
\end{equation}

By choosing $l'$ and then $\underline l$  large enough, we can guarantee that
\[
I:=(-\underline{l},\underline{H})\cap (-l'-c(T_3-T_1),L_2-c(T_3-T_1))\not= {\O}.
\]
We then
choose $V_0\in C([-\underline{l},\underline{H}])$ nonnegative and satisfy $0< V_0(x) \leq w_{-l',L_2}(x+c(T_3-T_1))$ for $x\in I$, $V_0(x)=0$ for $x\in [-\underline l, \underline H]\setminus I$, and consider the solution $w_{\underline{H}}$ of
 \begin{equation*} \begin{cases}
 w_t-dw_{xx}-cw_x=A\left(x\right)w-bw^2,\ &t>0,\ -\underline{l}<x<\underline{H},\\
 w(-\underline{l},t)=0,\ w\left(\underline{H},t\right)=0 &t>0,\\
 w(x,0)=V_0(x),\ &x\in [-\underline{l},\underline{H}].    	
 \end{cases}
 \end{equation*}It is well known that $w_{\underline{H}}(x,t)\to \psi_{-\underline{l},\underline{H}}(x)$ uniformly for $x\in [-\underline{l},\underline{H}]$ as $t\to +\infty$. By \eqref{new} and the choice of $V_0$, and the fact that $F(t)\geq \underline H$, we can apply the comparison principle to obtain
  $$\hat{V}(x,t-s)\geq w_{\underline{H}}(x,t),\quad \forall s>0,\ t>0,\ -\underline{l}<x<\underline{H}.
  $$
 Choosing $s_n\to \infty$, we have $w_{\underline{H}}(x,s_n)\leq \hat{V}(x,0)$ for $-\underline{l}<x<\underline{H}$.
 Letting $n\to \infty$, we obtain$$\psi_{-\underline{l},\underline{H}}(x)\leq \hat{V}(x,0)\quad \mbox{for}\ -\underline{l}<x<\underline{H}.$$
Hence we can use the strong maximum principle and the Hopf boundary lemma to conclude that $$\psi'_{-\underline{l},\underline{H}}(\underline{H})>\hat{V}_x(\underline{H},0)=-\frac{c}{\mu(A(\underline{H}))},$$
which contradicts \eqref{psiunderlinel}.
The proof is complete.\end{proof}

 From Lemmas \ref{barH} and \ref{underlineH} we immediately obtain the following result.

\begin{Theorem}\label{ht-ct=L0}
	 $\lim\limits_{t\to +\infty}[h(t)-ct]=L_0$.
\end{Theorem}

\subsection{Behaviour of $u(x,t)$} In this subsection, we completely determine the long-time behaviour of $u(x,t)$.

\begin{Lemma}\label{Mtgoeswuq}
	Assume that $(u(x,t),h(t))$ is the unique solution of problem \eqref{free boundary equation} and $0<c\leq c_0$. If there exists a constant $C_0$ such that $h(t)-ct>C_0$ for $t>0$, then
	$$\lim_{M\to+\infty}\limsup\limits_{t\to +\infty}\left[\max \limits_{x\in [0,ct-M]}\left|  u(x,t)-\frac{a}{b}\right| \right]=0.$$
\end{Lemma}

\begin{proof}
Assume the assertion of the lemma is false. Then we could find $\sigma>0$ and a sequence of positive numbers $M_n\to +\infty$ such that
\[
\limsup\limits_{t\to +\infty}\left[\max \limits_{x\in [0,ct-M_n]}\left|u(x,t)-\frac{a}{b}\right| \right]>\sigma \;\mbox{  for all $n\ge1$.} 
\]
 Thus there exist two sequences of numbers $\{x_n\}$ and $\{t_n\}$ such that	
 \[
 \mbox{$\lim\limits_{n\to +\infty}t_n=+\infty,\ 0\le x_n\le ct_n-M_n$, and}
 \]
\begin{equation}\label{jianu}
	\left| {u}(x_n,t_n)-\frac{a}{b}\right|>\sigma.
\end{equation}

By \eqref{cuab}, for any $\varepsilon>0$ small,
  there is a large $T=T_{\varepsilon}>0$ such that
  $${u}(0,t)>\frac{a}{b+\varepsilon}\ \text{for}\ t\geq T.
  $$
We may also assume that
\[
C_0>-cT \mbox{ and so } h(t)\geq  ct+C_0>c(t-T).
\] 

Since $0<c\leq c_0<2\sqrt{ad}$, the problem
	\begin{equation*}
	 \left\{\begin{array}{l}
-dv''-cv'=av-(b+\varepsilon)v^2,\quad x<0,\\
		v(0)=0
		\end{array} \right.	\end{equation*}
admits a unique solution $v_{\varepsilon}(x)$,  and it satisfies $v'_{\varepsilon}(x)<0$ and $v_{\varepsilon}(-\infty)=\frac{a}{b+\varepsilon}$ (see the proof of Proposition 2.1 in \cite{BDu2012}). Define $$\underline{u}_{\varepsilon}(x,t)=v_{\varepsilon}(x-c(t-T))\ \text{for}\ x\leq c(t-T) \ \text{and}\ t\geq T.$$
   Then for $t\geq T$, we have $$\frac{\partial \underline{u}_{\varepsilon}}{\partial t}-d\frac{\partial^2 \underline{u}_{\varepsilon}}{\partial x^2}
   \leq a \underline{u}_{\varepsilon}-b\underline{u}^2_{\varepsilon}\quad \text{for}\ 0< x< c(t-T),$$
   $$\underline{u}_{\varepsilon}(0,t)\leq \frac{a}{b+\varepsilon}<u(0,t)\quad \text{and}\quad \underline{u}_{\varepsilon}(c(t-T),t)=0<u(c(t-T),t).$$
  We can now use the comparison principle over $\{(x,t): 0\leq x\leq c(t-T),\ t\geq T\}$ to deduce 
  $${u}(x,t)\geq \underline{u}_{\varepsilon}(x,t) \mbox{ in this region}.
  $$
 Since $t_n\geq T$ and $0\leq x_n\leq ct_n-M_n\leq c(t_n-T)$ for all large $n$, we get $${u}(x_n,t_n)\geq \underline{u}_{\varepsilon}(x_n,t_n)=v_{\varepsilon}(x_n-c(t_n-T))\geq v_{\varepsilon}(-M_n+cT),$$
 and \begin{equation}\label{hatu}\liminf_{n\to +\infty}{u}(x_n,t_n)\ge \frac{a}{b+\varepsilon}.\end{equation}

On the other hand, applying the comparison principle, we conclude that
\begin{equation}\label{hatphi}
	{u}(x_n,t_n)\le \hat{\phi}(t_n)\to \frac{a}{b}\ \text{as}\ n\to +\infty,
\end{equation}
where $\hat{\phi}(t)$ is the unique solution of
\begin{equation*}
\left\{
\begin{array}{l}
\phi'(t)=a\phi-b\phi^2,\quad t>0,\\
\phi(0)={\lVert u_0\rVert}_{\infty}.\\
\end{array}
\right.
\end{equation*}
By virtue of \eqref{hatphi} and \eqref{jianu}, we obtain \begin{equation}\label{hatuxiao}
	{u}(x_n,t_n)< \frac{a}{b}-\sigma\ \text{for all large}\ n,
\end{equation}
which contradicts \eqref{hatu} if $\varepsilon>0$ is small enough. The proof is complete.\end{proof}

\begin{Theorem}\label{a/b}
		$\lim\limits_{t\to +\infty}\| u(\cdot,t)-v_{L_0}(\cdot-h(t)+L_0)\|_{L^{\infty}([0,h(t)])}=0$.
\end{Theorem}

\begin{proof}
 Define
	$$g(t):=h(t)-ct,\ V(x,t):=u(x+h(t),t)\ \text{for}\ t>0\ \text{and}\ -h(t)<x<0.$$
	For any sequence $\{t_n\}$ satisfying $t_n\to+\infty$, denote
	$$g_n(t)=g(t+t_n),\quad V_n(x,t)=V(x,t+t_n).$$
	Then due to Theorem \ref{ht-ct=L0} we can duplicate the demonstration in the proof of Theorem 3.13 in \cite{DuWeiZhou2018} to conclude that
	$$\lim_{t\to +\infty}V(\cdot,t)=v_{L_0}(\cdot+L_0)\ \text{in}\ C^{1+\alpha}_{loc}((-\infty,0])$$
	for some $\alpha\in(0,1).$
	It follows that
	\begin{equation}\label{ct-Mh(t)}
	\lim_{t\to+\infty}\left[\max \limits_{ct-M\le x\le h(t)}\left|  u(x,t)-v_{L_0}(x-h(t)+L_0)\right|\right]=0\ \mbox{ for every}\ M>0.
	\end{equation}
	Thanks to $v_{L_0}(-\infty)=\frac{a}{b}$, by Lemma \ref{Mtgoeswuq}, we have
	\begin{equation}\label{epislonM}
		\bar\epsilon(M):=\limsup_{t\to +\infty}\left[\max \limits_{0\le x\le ct-M}\left|u(x,t)-v_{L_0}(x-h(t)+L_0)\right| \right]\to 0\ \text{as}\ M\to +\infty.
	\end{equation}
	From \eqref{ct-Mh(t)} and \eqref{epislonM}, we deduce that
	$$\limsup_{t\to +\infty}\left[\max \limits_{0\le x\le h(t)}\left|u(x,t)-v_{L_0}(x-h(t)+L_0)\right| \right]=\bar\epsilon(M).$$
	Letting $M\to+\infty$, then the desired conclusion follows.\end{proof}

Clearly the conclusion in case (ii) of Theorem \ref{c_0>c} follows directly from Theorems \ref{ht-ct=L0} and \ref{a/b}.

\section{Proof of Theorem \ref{c_0=c} }

 If $h_{\infty}<+\infty$, it follows from Theorem \ref{0xiaoyucvanishing} that
	 $$\lim_{t\to +\infty}\|u(x,t)\|_{C([0,h(t)])}=0.$$

Next we consider the case $h_\infty=+\infty$.
Let $(\hat{u}(x,t),\bar{h}(t))$ be the unique solution of problem \eqref{free boundary equation} with $A(x-ct)$ replaced by $a$ and $\mu(A(h(t)-ct))$ replaced by $\mu(a)$. Since $A(x-ct)\le a$ and $\mu(A(\bar{h}(t)-ct))\leq \mu(a)$, by the comparison principle, we conclude that
	$$	0\le u(x,t)\le \hat{u}(x,t)\ \text{for}\ t\ge 0,\ x\in [0,h(t)],$$	and
	$$	h(t)\le \bar{h}(t)\ \text{for}\ t\ge 0.$$

 Since $h_{\infty}=+\infty$, necessarily $\overline{h}_{\infty}=+\infty$, and we conclude from Theorem 1.2 of \cite{DuMZhou2014} that 
 \begin{equation*}\label{disH}
\lim_{t\to+\infty}[\bar{h}(t)-c_0t]=\hat{H}
\end{equation*}
for some $\hat{H}\in \mathbb{R}^1$, and hence 
$$\bar{H}=\limsup_{t\to+\infty}[h(t)-c_0t]\leq \hat{H}<+\infty.
$$

We show below that 
$$
\underline{H}=\liminf_{t\to+\infty}[h(t)-c_0t]>-\infty.
$$
Firstly, we choose $T>0$ large so that $h_T=\min\{h(T), c_0T\}>\frac{\pi}{2}\sqrt{\frac{d}{a}}$. Fix $\widetilde{u}_0\in C^2([0,h_T])$ satisfying $0<\widetilde{u}_0(x) \le \min\{u(x,T), Q(x,T)\}$ for $x\in [0, h_T)$, $\widetilde u_0'(0)=0$, where $Q(x,t)=q_{c_0}(x-c_0t)$. Then the auxiliary free boundary problem
$$	\left\{
	\begin{aligned}
	&\widetilde{u}_t=d\widetilde{u}_{xx}+A(x-c_0t)\widetilde{u}-b\widetilde{u}^2,&t&>T,\ 0<x<\widetilde{h}(t),\\
	&\widetilde{u}_x(0,t)=\widetilde{u}(\widetilde{h}(t),t)=0,  &t&>T,     \\
	&\widetilde{h}'(t)=-\mu(A(\widetilde{h}(t)-c_0t))\widetilde{u}_x(\widetilde{h}(t),t),&t&>T, \\
	&\widetilde{h}(T)=h_{T},\ \widetilde{u}(x,T)=\widetilde{u}_0(x), &0&\leq x \leq h_{T},
	\end{aligned}
	\right.	
	$$ 
	has a unique  solution $(\widetilde{u}(x,t),\widetilde{h}(t))$. By the comparison principle, for $t>T$,
$$\widetilde{u}(x,t)\leq u(x,t)\ \mbox{for}\ x\in [0,\widetilde{h}(t)],\quad \widetilde{h}(t)\leq h(t),
$$
and
$$
\widetilde{u}(x,t)\leq Q(x,t)\ \mbox{for}\ x\in [0,\widetilde{h}(t)],\quad \widetilde{h}(t)\leq c_0t.
$$
Hence for $t>T$ and $x\in [0,\widetilde{h}(t)]$, we have $A(x-c_0t)=a$ and $\mu(A(\widetilde{h}(t)-c_0t))=\mu(a)$ in view of the definitions of $A(\xi)$ and $\mu(\zeta)$. It follows from this fact and $h_T>\frac{\pi}{2}\sqrt{\frac{d}{a}}$ that $\widetilde{h}_{\infty}=+\infty$ (see Theorem 3.4 in \cite{DuLin2010}) and $\displaystyle{\lim_{t\to+\infty}[\widetilde{h}(t)-c_0t]=\check{H}}$ for some $\check{H}\in\R^1$ (see Theorem 1.2 in \cite{DuMZhou2014}). Thus 
$$\underline{H}=\liminf_{t\to+\infty}[h(t)-c_0t]\geq \liminf_{t\to+\infty}[\widetilde{h}(t)-c_0t]=\check{H}>-\infty,
$$
as we wanted.

We now continue our discussion according to the following three possibilities: 
\[
\mbox{(I) $\bar{H}<0$,\ (II) $\bar{H}>0$, (III) $\bar{H}=0$.}
\]
In case (I), there exists $\bar{T}>0$ such that	
$$ 
h(t)-c_0t<0\ \text{for}\ t\ge \bar{T}.
$$
By the definitions of $A(\xi)$ and $\mu(\zeta)$, we see that for $t\ge \bar{T}$,
	$$ 
	A(x-c_0t)=a\ \mbox{for}\ x\in[0,h(t)],\ \mbox{and}\ \mu(A(h(t)-c_0t))=\mu(a).
	$$
	Therefore, from \cite{DuMZhou2014}, we see that 
$$\lim_{t\to +\infty}[h(t)-c_0t]= \bar H,\quad \lim_{t\to +\infty}\| u(\cdot,t)-q_{c_0}(\cdot-h(t))\|_{L^{\infty}([0,h(t)])}=0.
$$

In case (II), for any $L>0$, let $v_L$ denote the unique positive solution of
\begin{equation*}
\left\{
\begin{array}{l}
-dv''-c_0v'=A(x)v-bv^2,\quad -\infty<x<L,\\
\;\;\; v(L)=0.
\end{array}
\right.
\end{equation*}
By Proposition \ref{v_L} (ii),
 $$-\mu(A(L))v_L'(L)<c_0,\ \forall L>0.$$
 We may now take $L=\bar H$ and
 use the same argument as in the proof of Lemma \ref{barH} to deduce a contradiction. Therefore case (II) cannot occur.

In case (III), we have $\bar{H}=0$.   Choose $\overline{C}>0$ such that $-\overline{C}\leq h(t)-c_0t<\overline{C}$ for all $t\ge0$.
Let $\{t_n\}$ be an arbitrary positive sequence satisfying
 $t_n\to +\infty$. Define
$$\check{g}(t)=h(t)-c_0t-2\overline{C},\quad v^{\ast}(x,t)=u(x+c_0t+2\overline{C},t),$$
$$V^{\ast}_n(x,t)=v^{\ast}(x,t+t_n),\quad \check{g}_n(t)=\check{g}(t+t_n),$$
$$\tilde{y}=\frac{x}{\check{g}_n(t)}\quad \text{and}\quad w^{\ast}_n(\tilde{y},t)=V^{\ast}_n(x,t).$$
By the same argument as in the proof of Theorem \ref{ht-ct=L0}, we deduce that, by passing to a subsequence, as $n\to+\infty$, $$V^{\ast}_n\to \tilde{V}\ \text{in}\ C_{loc}^{1+\bar{\alpha},\frac{1+\bar{\alpha}}{2}}((-\infty,\check{G}(t)]\times\mathbb{R}^1),\quad \check{g_n}\to\check{G}\ \text{in}\ C_{loc}^{1+\frac{\bar{\alpha}}{2}}(\mathbb{R}^1),$$
where $\bar{\alpha}\in(0,1)$ and $(\tilde{V},\check{G})$ satisfies
\begin{equation*}
\left\{
\begin{aligned}
&V_t-dV_{xx}-c_0V_x=A(x+2\overline{C})V-bV^2,\ &t&\in \mathbb{R}^1,\ x\in (-\infty,\check{G}(t)),\\
&V\left(\check{G}(t),t\right)=0,\ -\mu(A(\check{G}(t)+2\overline{C})) V_x(\check{G}(t),t)=\check{G}'(t)+c_0,\ &t&\in \mathbb{R}^1.\end{aligned}
\right.
\end{equation*}
Let
$$\bar{V}(x,t)=\tilde{V}(x-2\overline{C},t)\ \text{and}\ \check{F}(t)=\check{G}(t)+2\overline{C}.$$
We observe that $\check{F}(t)\le0$, and hence $(\bar{V},\check{F})$ satisfies
\begin{equation*}\label{barVcF}
\left\{
\begin{aligned}
&V_t-dV_{xx}-c_0V_x=aV-bV^2,\ &t&\in \mathbb{R}^1,\ x\in (-\infty,\check{F}(t)),\\
&V\left(\check{F}(t),t\right)=0,\ -\mu(a) V_x(\check{F}(t),t)=\check{F}'(t)+c_0,\ &t&\in \mathbb{R}^1. \end{aligned}
\right.
\end{equation*}
  By the results in section 4.2 of \cite{DuMZHou2015}, we see that there exists $F_0\leq 0$ such that
  $$
  \check{F}(t)\equiv F_0,\quad \bar{V}(x,t)\equiv{q}_{c_0}(x-F_0)\ \text{for}\ x\in(-\infty,F_0],\ t\in\mathbb{R}^1.
  $$
  If  $F_0<0$, then we may repeat the argument in section 3.3 of \cite{DuMZhou2014} to conclude that
  $$\lim_{t\to +\infty}[h(t)-c_0t]=F_0<0,$$
  a contradiction to $\bar H=0$. Hence we always have $F_0=0$. Since $\{t_n\}$ is an arbitrary sequence converging to $+\infty$, this implies that
   $$\lim_{t\to +\infty}[h(t)-c_0t]=0,$$  
   and
   \begin{equation*}
  	\lim_{t\to +\infty}u(x+c_0t,t)=q_{c_0}(x)\ \text{in}\ C_{loc}^{0}(-\infty,0],
  \end{equation*}
  which infers that
  \begin{equation}\label{uxct}
  \lim_{t\to +\infty}u(x+h(t),t)=q_{c_0}(x)\ \text{in}\ C_{loc}^{0}(-\infty,0].
  \end{equation}
 For any given $M\ge0$, let 
 \[
 \hat{\epsilon}(M):=\limsup\limits_{t\to+\infty}\left[\max \limits_{0\le x\le c_0t-M}| u(x,t)-q_{c_0}(x-h(t))|\right];
 \]
  then we have
 $$
 \hat{\epsilon}(M)\le \limsup_{t\to+\infty}\left[\max_{0\le x\le c_0t-M}\left|u(x,t)-\frac{a}{b}\right|\right]+\limsup_{t\to+\infty}\left[\max_{0\le x\le c_0t-M}\left| \frac{a}{b}-q_{c_0}(x-h(t))\right|\right].
 $$
 Combining the fact $q_{c_0}(-\infty)=\frac{a}{b}$ and Lemma \ref{Mtgoeswuq}, we obtain
 \begin{equation}\label{eM}
 	\hat{\epsilon}(M)\to 0\ \text{as}\ M\to+\infty.
 \end{equation}
 Obviously,
 \begin{equation}\label{5lmu}
 \begin{split}
 &\limsup_{t\to+\infty}\left[\max_{0\le x\le h(t)}|u(x,t)-q_{c_0}(x-h(t))|\right]\\
 &\le \limsup_{t\to+\infty}\left[\max_{c_0t-M\le x\le h(t)}|u(x,t)-q_{c_0}(x-h(t))|\right]+\hat{\epsilon}(M)\\
 &=\hat{\epsilon}(M). \\	
 \end{split}    	
 \end{equation}
 Letting $M\to+\infty$ in \eqref{5lmu}, we deduce from  \eqref{eM} that
 $$\lim_{t\to+\infty}\left[\max_{0\le x\le h(t)}| u(x,t)-q_{c_0}(x-h(t))|\right]=0.$$
 This completes the proof.

\section{Proof of Theorem \ref{c_0<c}}

	Let $(\overline{u}(x,t),\bar{h}(t))$ denote the unique solution of problem \eqref{free boundary equation} with $A(x-ct)$ replaced by $a$ and $\mu(A(h(t)-ct))$ replaced by $\mu(a)$. By Theorems 3.3 and 4.2 in \cite{DuLin2010}, we see that there exists $\bar{h}_{\infty}\in (0,+\infty]$, such that $\lim\limits_{t\to +\infty}\bar{h}(t)=\bar{h}_{\infty}$ and the following alternative holds: \\
	Either
	
	(i)  $\bar{h}_{\infty}<+\infty$ and	$$	\lim_{t\to +\infty}\|\overline{u}(x,t)\|_{C([0,\bar{h}(t)])}=0;$$
	or
	
	(ii) $\bar{h}_{\infty}=+\infty,\ \lim\limits_{t\to +\infty}\overline{u}(t,x)=\frac{a}{b}$ uniformly for $x$ in any bounded subset of $[0,+\infty)$ and
	$$	\lim_{t\to +\infty}\frac{\bar{h}(t)}{t}=c_0.$$
	
Since $A(x-ct)\le a$ and  $\mu(A(\bar{h}(t)-ct))\leq \mu(a)$, by the comparison principle, we conclude that
	$$	0\le u(x,t)\le \overline{u}(x,t)\ \text{for}\ t\ge 0,\ x\in [0,h(t)],$$
	and
	$$	h(t)\le \bar{h}(t)\ \text{for}\ t\ge 0.$$
	
	If $h_{\infty}<+\infty$, by Theorem \ref{0xiaoyucvanishing}, we obtain $\displaystyle{\lim_{t\to +\infty}\|{u}(\cdot,t)\|_{C([0,h(t)])}}=0$.
	
	If $h_{\infty}=+\infty$, we have $\limsup\limits_{t\to +\infty}\frac{h(t)}{t}\le\lim\limits_{t\to +\infty}\frac{\bar{h}(t)}{t}= c_0$. Thanks to $c_0< c$, there exists $\bar{T}>0$ such that	$$ h(t)-ct<0\ \text{for}\ t\ge \bar{T}.
	 $$
	Hence,  for $t\ge \bar{T}$, we obtain
	$$ A(x-ct)=a\ \mbox{for}\ x\in[0,h(t)],\ \mbox{and}\ \mu(A(h(t)-ct))=\mu(a).$$
	Using Theorem 1.2 in \cite{DuMZhou2014}, we obtain
	$$\lim_{t\to +\infty}[h(t)-c_0t]= \check{C}\ \text{for some}\ \check{C}\in\R^1,$$
	and	$$\lim_{t\to +\infty}\| u(\cdot,t)-q_{c_0}(\cdot-h(t))\|_{L^{\infty}([0,h(t)])}=0.$$
The proof is complete.

\section{Proof of Theorem 1.5}

Define
$\mathcal{X}(h_0):=\{\phi \in C^2([0,h_0]) : \phi \text{ satisfies}\ \eqref{i}\}$.
Fix $\phi \in \mathcal{X}(h_0)$, and for $\sigma > 0$, let $(u_{\sigma}, h_{\sigma})$ denote the unique positive solution of problem \eqref{free boundary equation} with initial value $u_0 =\sigma\phi$. Clearly $u_0\in \mathcal{X}(h_0)$.
We complete the proof by several lemmas.

\begin{Lemma}\label{sxsd}
	$(i)$ If vanishing happens for $(u_{\sigma^1},h_{\sigma^1})$, then vanishing occurs when $0<\sigma\le\sigma^{1}$.
	
	$(ii)$ If spreading happens for $(u_{\sigma^1},h_{\sigma^1})$, then spreading also happens when $\sigma\geq \sigma^{1}$.
\end{Lemma}

\begin{proof}
The lemma follows from the comparison principle immediately.\end{proof}

Define
$$\sigma_{\ast}=\sigma_{\ast}(c,h_0)=\left\{\begin{array}{l} 0,\;  \mbox{ if } (h_\sigma)_\infty=+\infty \mbox{ for every } \sigma>0, \smallskip 
\\
\sup\big\{\sigma_{0}: (h_\sigma)_\infty<+\infty \text{ for}\ \sigma\in(0,\sigma_0]\big\},\; \mbox{ otherwise}.
\end{array}\right.
$$	
Clearly $\sigma_*\in [0,+\infty]$.

\begin{Lemma}\label{sad0}
	For any $\sigma>0$, if there exists $T\ge0$ such that $h_{\sigma}(T)\ge\frac{\pi}{2}\sqrt{\frac{d}{a}}$, then $(h_{\sigma})_{\infty}=+\infty$.
\end{Lemma}

\begin{proof}
To obtain a contradiction, suppose that $(h_{\sigma})_{\infty}<+\infty$. It follows that there exists $T^{\ast}>T$ such that $$h_{\sigma}(t)-ct<0\ \text{for all}\ t\ge T^{\ast}.$$
Therefore, for $t\ge T^{\ast}$,
$$A(x-ct)=a\ \text{for}\ x\in[0,h_{\sigma}(t)],\ \mbox{and}\ \mu(A(h_{\sigma}(t)-ct))=\mu(a).$$
Due to $h_{\sigma}^{\prime}(t)>0,$
$$h_{\sigma}(T^{\ast})>h_{\sigma}(T)\geq \frac{\pi}{2}\sqrt{\frac{d}{a}}.$$
By the arguments in the proof of Lemma 3.1 in \cite{DuLin2010}, we  derive a contradiction.
\end{proof}

\begin{Lemma}\label{tsad0}
 	If $h_0 \ge\frac{\pi}{2}\sqrt{\frac{d}{a}}$, then $\sigma_{\ast}=0$.
\end{Lemma}

\begin{proof}
By Lemma \ref{sad0}, we deduce that $(h_{\sigma})_{\infty}=+\infty$ for all $\sigma>0$,	and hence  $\sigma_{\ast}=0$.\end{proof}

\begin{Lemma}If $h_0<\frac{\pi}{2}\sqrt{\frac{d}{a}}$, then $\sigma_{\ast}\in (0, +\infty]$.\end{Lemma}

\begin{proof}
For any $\sigma>0$, let $(\bar{u}_{\sigma},\bar{h}_{\sigma})$ denote the unique solution of problem \eqref{free boundary equation} with $A(x-ct)$ replaced by $a$ and $\mu(A(h(t)-ct))$ replaced by $\mu(a)$, and initial value $u_0=\sigma\phi$. By Lemma 3.8 in \cite{DuLin2010},
 vanishing happens for $(\bar{u}_{\sigma},\bar{h}_{\sigma})$ for all small $\sigma>0$, and so $(\bar{h}_{\sigma})_{\infty}<+\infty$ for such $\sigma$. It follows from the comparison principle that $(h_{\sigma})_{\infty}<+\infty$ for all small $\sigma>0$. Thus $\sigma_{\ast}\in (0, +\infty]$. This completes the proof.\end{proof}

\begin{Lemma}\label{VorS}
	$(\mathrm{i})$ Vanishing happens for $(u_{\sigma},h_{\sigma})$ when $\sigma\le\sigma_{\ast}$.
	
	$(\mathrm{ii})$ Spreading happens for $(u_{\sigma},h_{\sigma})$ when $\sigma>\sigma_{\ast}$.
\end{Lemma}

\begin{proof}
If $\sigma_{\ast}=+\infty$, then there is nothing left to prove. We suppose next $\sigma_{\ast}\in[0,+\infty)$.
	
If $h_0 \ge\frac{\pi}{2}\sqrt{\frac{d}{a}}$, then $\sigma_{\ast}=0$ by Lemma \ref{tsad0}, and the desired conclusion follows trivially.
	
 Hereinafter, we assume that $h_0 <\frac{\pi}{2}\sqrt{\frac{d}{a}}$.
	To complete the proof, it suffices to show that vanishing happens when $\sigma=\sigma_{\ast}$.
Suppose on the contrary that spreading happens when $\sigma=\sigma_{\ast}$. Then  there exists $t_1>0$ such that
	$h_{\sigma_{\ast}}(t_1)>\frac{\pi}{2}\sqrt{\frac{d}{a}}+1$. Since the solution of problem \eqref{free boundary equation} depends continuously on its initial value, we deduce that
	 $$h_{\sigma_{\ast}-\epsilon}(t_1)>\frac{\pi}{2}\sqrt{\frac{d}{a}}$$
	 for all small $\epsilon>0$.
	 By Lemma \ref{sad0}, it follows that
	 $$\left(h_{\sigma_{\ast}-\epsilon}\right)_{\infty}=+\infty,$$
	 which implies that spreading happens for $(u_{\sigma_{\ast}-\epsilon},h_{\sigma_{\ast}-\epsilon})$. But this is a contradiction to the definition of $\sigma_{\ast}$. The proof is complete.\end{proof}

\end{document}